\documentclass[journal,onecolumn,12pt]{article}

\usepackage{amssymb}
\usepackage{caption2}
\usepackage{amsmath}
\usepackage{multirow}
\usepackage{amsthm}
\usepackage{graphicx}
\usepackage{CJK}
\usepackage{enumerate}
\usepackage{bm}

\usepackage{tikz}
\usepackage{tikz}
\usetikzlibrary{positioning,chains,fit,shapes,calc}

\newtheorem{theorem}{Theorem}[section]
\newtheorem{lemma}[theorem]{Lemma}
\newtheorem{proposition}[theorem]{Proposition}
\newtheorem{corollary}[theorem]{Corollary}

\newtheorem{remark}[theorem]{Remark}

\newcommand{\s}{\subseteq}

\begin{document}
\title{Dot-product graphs in finite fields}

\author{Chengfei Xie\thanks{Corresponding author. C. Xie is with the Institute of Mathematics and Interdisciplinary Sciences, Xidian University, Xi’an 710126, China (e-mail:cfxie@cnu.edu.cn). The research of C. Xie is supported by the National Natural Science Foundation of China under Grant No. 12401440.}
~and Gennian Ge\thanks{G. Ge is with the School of Mathematical Sciences, Capital Normal University, Beijing 100048, China (e-mail: gnge@zju.edu.cn). The research of G. Ge is supported by the National Key Research and Development Program of China under Grant 2020YFA0712100, the National
Natural Science Foundation of China under Grant 12231014, and Beijing Scholars Program.}}

\maketitle

\begin{abstract}
In this paper, we study the dot-product graphs in $\mathbb{F}_q^d$. We prove that if the size of the product of two adjacent sets is large enough, then  the set of dot-product graphs has positive density. Our method is based on finite field Fourier analytic techniques.
\smallskip
\end{abstract}
\medskip

\noindent {{\it Keywords\/}: dot product, graph, finite field, Fourier transform}

\smallskip

\noindent {{\it AMS subject classifications\/}: 11T30, 11T23, 52C10}

\section{Introduction}
The Erd\H{o}s distinct distance problem asks the minimum number of distinct distances determined by $n$ points in the plane. Erd\H{o}s conjectured that this number is at least $cn/\sqrt{\log n}$ for some constant $c$. In a significant breakthrough, Guth and Katz \cite{MR3272924} proved that the number of distinct distances is at least $cn/\log n$. The continuous analog of this problem is the Falconer's conjecture, which states that if the Hausdorff dimension of a set $A\subseteq\mathbb{R}^d$ is strictly larger than $d/2$, then the set of Euclidean distances between pairs of points in $A$ has positive Lebesgue measure. We refer the readers to \cite{MR4201782}, \cite{MR4297185} and \cite{MR4055179} for the best known results.

Let $\mathbb{F}_q$ be a field with $q$ elements. We write $X\lesssim Y$ if there exists a constant $C$ such that $X\leq CY$, and write $X\sim Y$ if $X\lesssim Y$ and $Y\lesssim X$. In the setting of finite fields, the Erd\H{o}s-Falconer distance problem asks the minimum size of $A\subseteq\mathbb{F}_q^d$ to ensure that the distance set
$$
\Delta(A):=\{\|x-y\|:x, y\in A\}
$$
contains a positive proportion of $\mathbb{F}_q$, where $\|x\|:=x_1^2+x_2^2+\cdots+x_d^2$ for $x=(x_1, x_2, \ldots, x_d)\in\mathbb{F}_q^d$. In \cite{MR2336319}, Iosevich and Rudnev showed that if $|A|\geq Cq^{(d+1)/2}$ for some constant $C$, then $\Delta(A)=\mathbb{F}_q$. They also proved that $|A|\geq Cq^{d/2}$ is a necessary condition in general.

Similar to the distance, we define the dot-product set
$$
\Pi(A):=\{x\cdot y:x, y\in A\}
$$
for $A\subseteq\mathbb{F}_q^d$. It is natural to ask the minimum size of $A\subseteq\mathbb{F}_q^d$ to ensure that $\Pi(A)$ contains a positive proportion of $\mathbb{F}_q$.
Hart et al. \cite{MR2775806} proved that if $A\subseteq\mathbb{F}_q^d$ with $|A|>q^{\frac{d+1}{2}}$, then $\Pi(A)\supseteq\mathbb{F}_q^*$. They also showed that the exponent $(d+1)/2$ cannot be improved in general. In \cite{MR3754899}, Covert and Senger considered pairs of dot products.  For $A\subseteq\mathbb{F}_q^d$, define
$$
\Pi(A\times A\times A):=\{(x\cdot y, y\cdot z):x, y, z\in E\}.
$$
Covert and Senger showed that if $|A|\gtrsim q^{\frac{d+1}{2}}$, then $\Pi(A\times A\times A)\supseteq\left(\mathbb{F}_q^*\right)^2$. Blevins et al. \cite{MR4539819} proved a more general result, which is called dot product chains. Let
$$
\Pi\left(A^{k+1}\right):=\{(x_1\cdot x_2, x_2\cdot x_3, \ldots, x_k\cdot x_{k+1}):x_i\in A, 1\leq i\leq k+1\}.
$$
If $|A|\gtrsim q^{\frac{d+k-1}{2}}$, then $\Pi\left(A^{k+1}\right)\supseteq\left(\mathbb{F}_q^*\right)^{k}$. On the other hand, Chapman et al. \cite{MR2917133} considered the $k$-simplices determined by dot products in $A$ and proved that if $A\subseteq\mathbb{F}_q^d$ and $|A|\gtrsim q^{\frac{d+k}{2}}$, then $A$ determines a positive proportion of all $k$-simplices.

We consider more general objects called dot-product graphs. Let $G=G(V, E)$ be a connected simple graph, where the vertex set $V=\{v_1, v_2, \ldots, v_k\}$ and the edge set $E\subseteq {V\choose 2}$. For every $\{v_i, v_j\}\in E$, we assign a value $\alpha({\{v_i, v_j\}})$ in $\mathbb{F}_q$, i.e. a function $\alpha: E\rightarrow\mathbb{F}_q$. We also view $\alpha$ as a row vector in $\left(\mathbb{F}_q\right)^E$, whose columns are indexed by elements in $E$, i.e. its $\{v_i, v_j\}$'s column is $\alpha({\{v_i, v_j\}})$ for $\{v_i, v_j\}\in E$. For $S\subseteq\left(\mathbb{F}_q^d\right)^k$, we define a set $\Pi_G(S)$, where a function $\alpha: E\rightarrow\mathbb{F}_q$ belongs to $\Pi_G(S)$ if and only if there exist $(x_1, x_2, \ldots, x_k)\in S$, such that for every $\{v_i, v_j\}\in E$, we have $x_i\cdot x_j=\alpha(\{v_i,v_j\})$. Using these notations, some of above results can be formulated as follows.
\begin{theorem}[\cite{MR2917133}]\label{simplexjiu}
Let $G$ be a complete graph with $k+1$ vertices. If $A\subseteq\mathbb{F}_q^d$ with $|A|\gtrsim q^{\frac{d+k}{2}}$, then $\left|\Pi_G\left(A^{k+1}\right)\right|\gtrsim q^{k+1\choose 2}$.
\end{theorem}
\begin{theorem}[\cite{MR4539819}]\label{chainjiu}
Let $G$ be a path with $k+1$ vertices. If $A\subseteq\mathbb{F}_q^d$ with $|A|\gtrsim q^{\frac{d+k-1}{2}}$, then $\Pi_G\left(A^{k+1}\right)\supseteq\left(\mathbb{F}_q^*\right)^{k}$.
\end{theorem}

In this paper, we generalize Theorems \ref{chainjiu} and \ref{simplexjiu}, and we have the following theorems.
\begin{theorem}\label{tuxin}
For integers $d$ and $k$ with $2\leq k\leq d+1$, there exists $C=C(k)$ such that the following holds. Let $G=G(V, E)$ be a connected simple graph, where $V=\{v_1, v_2, \ldots, v_k\}$ and $E\subseteq{V\choose 2}$. If $A_1, A_2, \ldots, A_{k}\subseteq\mathbb{F}_q^d$ satisfy $|A_i||A_{j}|\geq Cq^{d+k-1}$ for every $\{v_i, v_j\}\in E$ and $S\subseteq\prod_{i=1}^{k}A_i$ with $|S|\sim\prod_{i=1}^{k}|A_i|$, then
$$
\left|\Pi_{G}(S)\right|\gtrsim q^{|E|}.
$$
In other words, $\Pi_{G}(S)$ contains a positive proportion of elements of $\left(\mathbb{F}_q\right)^E$.

In particular,
$$
\left|\Pi_{G}(A_1\times A_2\times\cdots\times A_k)\right|\gtrsim q^{|E|}.
$$
\end{theorem}
\begin{theorem}\label{shuxin}
For integers $d\geq1$ and $k\geq2$, there exists $C=C(k)$ such that the following holds. Let $G=G(V, E)$ be a tree with $V=\{v_1, v_2, \ldots, v_{k+1}\}$. If $A_1, A_2, \ldots, A_{k+1}\subseteq\mathbb{F}_q^d$ and $|A_i||A_{j}|\geq Cq^{d+k-1}$ for every $\{v_i ,v_j\}\in E$, then $\Pi_G\left(A_1\times A_2\times\cdots\times A_{k+1}\right)\supseteq\left(\mathbb{F}_q^*\right)^{k}$.
\end{theorem}
\begin{remark}
Theorems \ref{tuxin} and \ref{shuxin} allow points in different sets $A_1, A_2, \ldots, A_{k}$ rather than a single set $A$. Moreover,  Theorem \ref{tuxin} generalizes Theorem \ref{simplexjiu} to general graphs, and Theorem \ref{shuxin} generalizes Theorem \ref{chainjiu} to trees.
\end{remark}

\section{Preliminaries}
Let $f:\mathbb{F}_q^d\rightarrow\mathbb{C}$ be a function and $\chi$ be a nontrivial additive character on $\mathbb{F}_q$. The Fourier transform of $f$ with respect to $\chi$ is given by
$$
\widehat{f}(k)=q^{-d}\sum_{x\in\mathbb{F}_q^d}\chi(-x\cdot k)f(x)
$$
where $k\in \mathbb{F}_q^d$.
And the Fourier inversion theorem is given by
$$
f(x)=\sum_{k\in\mathbb{F}_q^d}\chi(x\cdot k)\widehat{f}(k).
$$
We also need the Plancherel theorem
$$
\sum_{k\in\mathbb{F}_q^d}\left|\widehat{f}(k)\right|^2=q^{-d}\sum_{x\in\mathbb{F}_q^d}\left|f(x)\right|^2.
$$

\section{Proof of Theorem \ref{tuxin}}
In this section, we prove Theorem \ref{tuxin}.

For $A\subseteq\mathbb{F}_q^d$ and $x_1, x_2, \ldots, x_\ell\in\mathbb{F}_q^d$, we define
$$
\Pi_{x_1, x_2, \ldots, x_\ell}\left(A\right)=\{(x_1\cdot y, x_2\cdot y, \ldots, x_\ell\cdot y)\in \mathbb{F}_q^\ell:y\in A\}.
$$
Moreover, for $A\subseteq\mathbb{F}_q^d$, $x_1, x_2, \ldots, x_\ell\in\mathbb{F}_q^d$, and $t_1, t_2, \ldots, t_\ell\in\mathbb{F}_q$, define
$$
\Pi_{x_1, x_2, \ldots, x_\ell}^A\left(t_1, t_2, \ldots, t_\ell\right)=|\{y\in A:x_i\cdot y=t_i, 1\leq i\leq\ell\}|.
$$
If $A$ is a subset of $\mathbb{F}_q^d$, we also view $A$ as its indicator function, i.e. $A(x)=1$ if $x\in A$ and $A(x)=0$ for other $x$ in $\mathbb{F}_q^d$.

\begin{lemma}\label{yinli}
Let $A_1, A_2, \ldots, A_\ell$ and $A$ be subsets of $\mathbb{F}_q^d$. We have
$$
\sum_{(x_1, x_2, \ldots, x_\ell)\in A_1\times A_2\times \cdots\times A_\ell}\sum_{t_1, t_2, \ldots, t_{\ell}\in\mathbb{F}_q}\left(\Pi_{x_1, x_2, \ldots, x_\ell}^A\left(t_1, t_2, \ldots, t_\ell\right)\right)^2\leq\frac{|A|^2\prod_{i=1}^\ell|A_i|}{q^\ell}+q^d|A|\prod_{i=1}^\ell|A_i|\left(\sum_{j=1}^{\ell}\frac{1}{|A_j|}\right).
$$
\end{lemma}
\begin{proof}
Given $A$ and $(x_1, x_2, \ldots, x_\ell)\in A_1\times A_2\times \cdots\times A_\ell$, $\Pi_{x_1, x_2, \ldots, x_\ell}^A\left(t_1, t_2, \ldots, t_\ell\right)$ is a function from $\mathbb{F}_q^\ell$ to $\mathbb{C}$. Let $\chi$ be a non-trivial additive character on $\mathbb{F}_q$. By the orthogonality of $\chi$, we have
\begin{equation}
  \begin{split}
  \Pi_{x_1, x_2, \ldots, x_\ell}^A\left(t_1, t_2, \ldots, t_\ell\right)=&\sum_{y\in A}\prod_{i=1}^\ell\left(q^{-1}\sum_{s_i\in \mathbb{F}_q}\chi(s_i(t_i-x_i\cdot y))\right)\\
=&q^{-\ell}\sum_{y\in \mathbb{F}_q^d}A(y)\prod_{i=1}^\ell\left(\sum_{s_i\in \mathbb{F}_q}\chi(s_i(t_i-x_i\cdot y))\right).\\
  \end{split}
\end{equation}
And hence
\begin{equation}
  \begin{split}
  &\widehat{\Pi}_{x_1, x_2, \ldots, x_\ell}^A\left(s_1, s_2, \ldots, s_\ell\right)\\
  =&q^{-\ell}\sum_{t_1, t_2, \ldots, t_\ell\in \mathbb{F}_q}\chi(-s_1t_1-s_2t_2-\cdots-s_\ell t_\ell)\Pi_{x_1, x_2, \ldots, x_\ell}^A\left(t_1, t_2, \ldots, t_\ell\right)\\
=&q^{-2\ell}\sum_{t_1, t_2, \ldots, t_\ell\in \mathbb{F}_q}\chi(-s_1t_1-s_2t_2-\cdots-s_\ell t_\ell)\sum_{y\in \mathbb{F}_q^d}A(y)\prod_{i=1}^\ell\left(\sum_{s'_i\in \mathbb{F}_q}\chi(s'_i(t_i-x_i\cdot y))\right)\\
=&q^{-2\ell}\sum_{y\in \mathbb{F}_q^d}A(y)\sum_{t_1, t_2, \ldots, t_\ell\in \mathbb{F}_q}\prod_{i=1}^\ell\chi(-s_it_i)\left(\sum_{s'_i\in \mathbb{F}_q}\chi(s'_i(t_i-x_i\cdot y))\right)\\
=&q^{-2\ell}\sum_{y\in \mathbb{F}_q^d}A(y)\prod_{i=1}^\ell\sum_{s'_i\in \mathbb{F}_q}\chi(-s'_ix_i\cdot y))\sum_{t_i\in \mathbb{F}_q}\chi((s'_i-s_i)t_i).\\
  \end{split}
\end{equation}
If $s_i\neq s'_i$ for some $1\leq i\leq \ell$, then $\sum_{t_i\in \mathbb{F}_q}\chi((s'_i-s_i)t_i)=0$. Thus
\begin{equation}\label{fuliye}
  \begin{split}
  &\widehat{\Pi}_{x_1, x_2, \ldots, x_\ell}^A\left(s_1, s_2, \ldots, s_\ell\right)\\
  =&q^{-2\ell}\sum_{y\in \mathbb{F}_q^d}A(y)\prod_{i=1}^\ell\sum_{s'_i\in \mathbb{F}_q}\chi(-s'_ix_i\cdot y))\sum_{t_i\in \mathbb{F}_q}\chi((s'_i-s_i)t_i)\\
=&q^{-2\ell}\sum_{y\in \mathbb{F}_q^d}A(y)\prod_{i=1}^\ell \chi(-s_ix_i\cdot y))\cdot q\\
=&q^{-\ell}\sum_{y\in \mathbb{F}_q^d}A(y)\chi((-s_1x_1-s_2x_2-\cdots-s_\ell x_\ell)\cdot y))\\
=&q^{d-\ell}\widehat{A}(s_1x_1+s_2x_2+\cdots+s_\ell x_\ell).\\
  \end{split}
\end{equation}

In order to prove the lemma, We use induction on $\ell$. When $\ell=1$, Equation (\ref{fuliye}) becomes
$$
\widehat{\Pi}_{x_1}^A\left(s_1\right)=q^{d-1}\widehat{A}(s_1x_1).
$$
It follows that
\begin{equation}
  \begin{split}
\sum_{s_1\in \mathbb{F}_q}\sum_{x_1\in A_1}\left|\widehat{\Pi}_{x_1}^A\left(s_1\right)\right|^2=&q^{2(d-1)}\sum_{s_1\in \mathbb{F}_q}\sum_{x_1\in A_1}\left|\widehat{A}(s_1x_1)\right|^2\\
=&I+II,
  \end{split}
\end{equation}
where $I$ is the case that $s_1=0$ and $II$ is the case that $s_1\neq0$.
\begin{equation}
  \begin{split}
I=&q^{2(d-1)}\sum_{x_1\in A_1}\left|\widehat{A}(0)\right|^2\\
 =&q^{2(d-1)}\sum_{x_1\in A_1}\left|q^{-d}\sum_{y\in \mathbb{F}_q^d}\chi(0)A(y)\right|^2\\
 =&q^{-2}|A|^2|A_1|.
  \end{split}
\end{equation}
and
\begin{equation}
  \begin{split}
II=&q^{2(d-1)}\sum_{s_1\neq0}\sum_{x_1\in A_1}\left|\widehat{A}(s_1x_1)\right|^2\\
 =&q^{2(d-1)}\sum_{s_1\neq0}\sum_{x_1\in \mathbb{F}_q^d}\left|\widehat{A}(s_1x_1)\right|^2A_1(x_1)\\
 =&q^{2(d-1)}\sum_{s_1\neq0}\sum_{m\in \mathbb{F}_q^d}\left|\widehat{A}(m)\right|^2A_1\left(\frac{m}{s_1}\right)\\
 =&q^{2(d-1)}\sum_{m\in \mathbb{F}_q^d}\left|\widehat{A}(m)\right|^2\sum_{s_1\neq0}A_1\left(\frac{m}{s_1}\right).\\
  \end{split}
\end{equation}
Note that $\sum_{s_1\neq0}A_1\left(\frac{m}{s_1}\right)\leq q$ for every $m$ and using Plancherel theorem we have
$$
II\leq q^{2d-1}\sum_{m\in \mathbb{F}_q^d}\left|\widehat{A}(m)\right|^2=q^{2d-1}\cdot q^{-d}\sum_{y\in \mathbb{F}_q^d}\left|{A}(y)\right|^2=q^{d-1}|A|.
$$
Applying Plancherel theorem again, it follows that
$$
\sum_{t_1\in \mathbb{F}_q}\sum_{x_1\in A_1}\left|{\Pi}_{x_1}^A\left(t_1\right)\right|^2=q\sum_{s_1\in \mathbb{F}_q}\sum_{x_1\in A_1}\left|\widehat{\Pi}_{x_1}^A\left(s_1\right)\right|^2\leq q^{-1}|A|^2|A_1|+q^{d}|A|,
$$
concluding the proof for $\ell=1$.

Now suppose the conclusion holds for $\ell-1$ and consider the case $\ell$. Applying Equation (\ref{fuliye}) again, we have
\begin{equation}
  \begin{split}
&\sum_{(x_1, x_2, \ldots, x_\ell)\in A_1\times A_2\times \cdots\times A_\ell}\sum_{s_1, s_2, \ldots, s_{\ell}\in\mathbb{F}_q}\left|\widehat{\Pi}_{x_1, x_2, \ldots, x_\ell}^A\left(s_1, s_2, \ldots, s_\ell\right)\right|^2\\
=&q^{2(d-\ell)}\sum_{(x_1, x_2, \ldots, x_\ell)\in A_1\times A_2\times \cdots\times A_\ell}\sum_{s_1, s_2, \ldots, s_{\ell}\in\mathbb{F}_q}\left|\widehat{A}(s_1x_1+s_2x_2+\cdots+s_\ell x_\ell)\right|^2\\
=&III+IV,
  \end{split}
\end{equation}
where $III$ is the case that $s_k=0$ and $IV$ is the case that $s_k\neq0$.
\begin{equation}
  \begin{split}
III=&q^{2(d-\ell)}\sum_{(x_1, x_2, \ldots, x_\ell)\in A_1\times A_2\times \cdots\times A_\ell}\sum_{s_1, s_2, \ldots, s_{\ell-1}\in\mathbb{F}_q}\left|\widehat{A}(s_1x_1+s_2x_2+\cdots+s_{\ell-1} x_{\ell-1})\right|^2\\
=&q^{2(d-\ell)}|A_\ell|\sum_{(x_1, x_2, \ldots, x_{\ell-1})\in A_1\times A_2\times \cdots\times A_{\ell-1}}\sum_{s_1, s_2, \ldots, s_{\ell-1}\in\mathbb{F}_q}\left|\widehat{A}(s_1x_1+s_2x_2+\cdots+s_{\ell-1} x_{\ell-1})\right|^2\\
=&q^{2(d-\ell)}|A_\ell|q^{-2(d-\ell+1)}\sum_{(x_1, x_2, \ldots, x_{\ell-1})\in A_1\times A_2\times \cdots\times A_{\ell-1}}\sum_{s_1, s_2, \ldots, s_{\ell-1}\in\mathbb{F}_q}\left|\widehat{\Pi}_{x_1, x_2, \ldots, x_{\ell-1}}^A\left(s_1, s_2, \ldots, s_{\ell-1}\right)\right|^2\\
=&q^{-2}|A_\ell|\sum_{(x_1, x_2, \ldots, x_{\ell-1})\in A_1\times A_2\times \cdots\times A_{\ell-1}}\sum_{s_1, s_2, \ldots, s_{\ell-1}\in\mathbb{F}_q}\left|\widehat{\Pi}_{x_1, x_2, \ldots, x_{\ell-1}}^A\left(s_1, s_2, \ldots, s_{\ell-1}\right)\right|^2.\\
  \end{split}
\end{equation}
Applying Plancherel theorem and the inductive hypothesis, we obtain
\begin{equation}
  \begin{split}
&\sum_{(x_1, x_2, \ldots, x_{\ell-1})\in A_1\times A_2\times \cdots\times A_{\ell-1}}\sum_{s_1, s_2, \ldots, s_{\ell-1}\in\mathbb{F}_q}\left|\widehat{\Pi}_{x_1, x_2, \ldots, x_{\ell-1}}^A\left(s_1, s_2, \ldots, s_{\ell-1}\right)\right|^2\\
=&\sum_{(x_1, x_2, \ldots, x_{\ell-1})\in A_1\times A_2\times \cdots\times A_{\ell-1}}\sum_{t_1, t_2, \ldots, t_{\ell-1}\in\mathbb{F}_q}q^{-\ell+1}\left|{\Pi}_{x_1, x_2, \ldots, x_{\ell-1}}^A\left(t_1, t_2, \ldots, t_{\ell-1}\right)\right|^2\\
\leq&q^{-\ell+1}\left(\frac{|A|^2\prod_{i=1}^{\ell-1}|A_i|}{q^{\ell-1}}+q^d|A|\prod_{i=1}^{\ell-1}|A_i|\left(\sum_{j=1}^{\ell-1}\frac{1}{|A_j|}\right)\right).\\
  \end{split}
\end{equation}
Therefore,
\begin{equation}
  \begin{split}
III\leq&\frac{|A|^2\prod_{i=1}^{\ell}|A_i|}{q^{2\ell}}+q^{d-\ell-1}|A|\prod_{i=1}^{\ell}|A_i|\left(\sum_{j=1}^{\ell-1}\frac{1}{|A_j|}\right)\\
\leq&\frac{|A|^2\prod_{i=1}^{\ell}|A_i|}{q^{2\ell}}+q^{d-\ell}|A|\prod_{i=1}^{\ell}|A_i|\left(\sum_{j=1}^{\ell-1}\frac{1}{|A_j|}\right).\\
  \end{split}
\end{equation}
For $IV$, we have
\begin{equation}
  \begin{split}
  IV=&q^{2(d-\ell)}\sum_{(x_1, x_2, \ldots, x_\ell)\in A_1\times A_2\times \cdots\times A_\ell}\sum_{\substack{s_1, s_2, \ldots, s_{\ell-1}\in\mathbb{F}_q\\s_\ell\neq0}}\left|\widehat{A}(s_1x_1+s_2x_2+\cdots+s_\ell x_\ell)\right|^2\\
  =&q^{2(d-\ell)}\sum_{(x_1, x_2, \ldots, x_{\ell-1})\in A_1\times A_2\times \cdots\times A_{\ell-1}}\sum_{\substack{s_1, s_2, \ldots, s_{\ell-1}\in\mathbb{F}_q\\s_\ell\neq0}}\sum_{x_\ell\in\mathbb{F}_q^d}A(x_\ell)\left|\widehat{A}(s_1x_1+s_2x_2+\cdots+s_\ell x_\ell)\right|^2\\
    =&q^{2(d-\ell)}\sum_{(x_1, x_2, \ldots, x_{\ell-1})\in A_1\times A_2\times \cdots\times A_{\ell-1}}\sum_{\substack{s_1, s_2, \ldots, s_{\ell-1}\in\mathbb{F}_q\\s_\ell\neq0}}\sum_{m\in\mathbb{F}_q^d}A(s_\ell^{-1}m+s_1x_1+s_2x_2+\cdots+s_{\ell-1} x_{\ell-1})\left|\widehat{A}(m)\right|^2.\\
  \end{split}
\end{equation}
Since
$$
\sum_{\substack{s_1, s_2, \ldots, s_{\ell-1}\in\mathbb{F}_q\\s_\ell\neq0}}A(s_\ell^{-1}m+s_1x_1+s_2x_2+\cdots+s_{\ell-1} x_{\ell-1})\leq q^\ell,
$$
it follows that
\begin{equation}
  \begin{split}
  IV\leq&q^{2d-\ell}\sum_{(x_1, x_2, \ldots, x_{\ell-1})\in A_1\times A_2\times \cdots\times A_{\ell-1}}\sum_{m\in\mathbb{F}_q^d}\left|\widehat{A}(m)\right|^2\\
  =&q^{2d-\ell}\prod_{i=1}^{\ell-1}|A_i|\cdot \sum_{m\in\mathbb{F}_q^d}\left|\widehat{A}(m)\right|^2\\
  =&q^{2d-\ell}\prod_{i=1}^{\ell-1}|A_i|\cdot q^{-d}\sum_{y\in\mathbb{F}_q^d}\left|{A}(y)\right|^2\\
  =&q^{d-\ell}|A|\prod_{i=1}^{\ell-1}|A_i|.\\
  \end{split}
\end{equation}
Thus
\begin{equation}
  \begin{split}
III+IV\leq&\frac{|A|^2\prod_{i=1}^{\ell}|A_i|}{q^{2\ell}}+q^{d-\ell}|A|\prod_{i=1}^{\ell}|A_i|\left(\sum_{j=1}^{\ell-1}\frac{1}{|A_j|}\right)+q^{d-\ell}|A|\prod_{i=1}^{\ell-1}|A_i|\\
=&\frac{|A|^2\prod_{i=1}^{\ell}|A_i|}{q^{2\ell}}+q^{d-\ell}|A|\prod_{i=1}^{\ell}|A_i|\left(\sum_{j=1}^{\ell}\frac{1}{|A_j|}\right).\\
  \end{split}
\end{equation}
Finally applying Plancherel theorem, we conclude that
\begin{equation}
  \begin{split}
&\sum_{(x_1, x_2, \ldots, x_\ell)\in A_1\times A_2\times \cdots\times A_\ell}\sum_{t_1, t_2, \ldots, t_{\ell}\in\mathbb{F}_q}\left(\Pi_{x_1, x_2, \ldots, x_\ell}^A\left(t_1, t_2, \ldots, t_\ell\right)\right)^2\\
=&q^\ell\sum_{(x_1, x_2, \ldots, x_\ell)\in A_1\times A_2\times \cdots\times A_\ell}\sum_{s_1, s_2, \ldots, s_{\ell}\in\mathbb{F}_q}\left|\widehat{\Pi}_{x_1, x_2, \ldots, x_\ell}^A\left(s_1, s_2, \ldots, s_\ell\right)\right|^2\\
\leq&\frac{|A|^2\prod_{i=1}^{\ell}|A_i|}{q^{\ell}}+q^{d}|A|\prod_{i=1}^{\ell}|A_i|\left(\sum_{j=1}^{\ell}\frac{1}{|A_j|}\right).
  \end{split}
\end{equation}
It completes the proof.
\end{proof}

\begin{theorem}\label{dingli}
For integers $d$ and $\ell$ with $2\leq \ell\leq d$, there exists $C=C(\ell)$ such that the following holds. Given $A_1, A_2, \ldots, A_{\ell}, A_{\ell+1}\subseteq\mathbb{F}_q^d$, let $S\subseteq\prod_{i=1}^{\ell+1}A_i$ with $|S|\sim\prod_{i=1}^{\ell+1}|A_i|$. Define
$$
S'=\{(x_1, x_2, \ldots, x_{\ell})\in \prod_{i=1}^{\ell}A_i:(x_1, x_2, \ldots, x_\ell, x_{\ell+1})\in S \text{ for some } x_{\ell+1}\in A_{\ell+1}\}.
$$
And for every $(x_1, x_2, \ldots, x_\ell)\in S'$ we define
$$
S(x_1, x_2, \ldots, x_\ell)=\{x_{\ell+1}\in A_{\ell+1}: (x_1, x_2, \ldots, x_\ell, x_{\ell+1})\in S\}.
$$
If $|A_i||A_{\ell+1}|\geq Cq^{d+\ell}$ for $1\leq i\leq\ell$, then
$$
\frac{1}{|S'|}\sum_{(x_1, x_2, \ldots, x_\ell)\in S'}\left|\Pi_{x_1, x_2, \ldots, x_\ell}\left(S(x_1, x_2, \ldots, x_\ell)\right)\right|\gtrsim q^\ell,
$$
where for $(x_1, x_2, \ldots, x_\ell)\in S'$, $\Pi_{x_1, x_2, \ldots, x_\ell}\left(S(x_1, x_2, \ldots, x_\ell)\right)=\{(x_1\cdot x_{\ell+1}, x_2\cdot x_{\ell+1}, \ldots, x_\ell\cdot x_{\ell+1})\in \mathbb{F}_q^\ell:x_{\ell+1}\in S(x_1, x_2, \ldots, x_\ell)\}$.
\end{theorem}
\begin{proof}
By the definition, we have
$$
\Pi_{x_1, x_2, \ldots, x_\ell}^{S(x_1, x_2, \ldots, x_{\ell})}\left(t_1, t_2, \ldots, t_\ell\right)\leq\Pi_{x_1, x_2, \ldots, x_\ell}^{A_{\ell+1}}\left(t_1, t_2, \ldots, t_\ell\right).
$$
Applying Cauchy-Schwartz inequality, we have
\begin{equation}
  \begin{split}
|S|^2=&\left(\sum_{(x_1, x_2, \ldots, x_{\ell})\in S'}\sum_{t_1, t_2, \ldots, t_{\ell}\in \mathbb{F}_q}\Pi_{x_1, x_2, \ldots, x_\ell}^{S(x_1, x_2, \ldots, x_{\ell})}\left(t_1, t_2, \ldots, t_\ell\right)\right)^2\\
\leq&\left(\sum_{(x_1, x_2, \ldots, x_{\ell})\in S'}\left|\Pi_{x_1, x_2, \ldots, x_\ell}\left(S(x_1, x_2, \ldots, x_{\ell})\right)\right|\right)\cdot\\
&\left(\sum_{(x_1, x_2, \ldots, x_{\ell})\in S'}\sum_{t_1, t_2, \ldots, t_{\ell}\in \mathbb{F}_q}\left(\Pi_{x_1, x_2, \ldots, x_\ell}^{S(x_1, x_2, \ldots, x_{\ell})}\left(t_1, t_2, \ldots, t_\ell\right)\right)^2\right)\\
\leq&\left(\sum_{(x_1, x_2, \ldots, x_{\ell})\in S'}\left|\Pi_{x_1, x_2, \ldots, x_\ell}\left(S(x_1, x_2, \ldots, x_{\ell})\right)\right|\right)\cdot\\
&\left(\sum_{(x_1, x_2, \ldots, x_{\ell})\in A_1\times A_2\times\cdots\times A_\ell}\sum_{t_1, t_2, \ldots, t_{\ell}\in \mathbb{F}_q}\left(\Pi_{x_1, x_2, \ldots, x_\ell}^{A_{\ell+1}}\left(t_1, t_2, \ldots, t_\ell\right)\right)^2\right)\\
\leq&\left(\sum_{(x_1, x_2, \ldots, x_{\ell})\in S'}\left|\Pi_{x_1, x_2, \ldots, x_\ell}\left(S(x_1, x_2, \ldots, x_{\ell})\right)\right|\right)\cdot\\
&\left(\frac{|A_{\ell+1}|^2\prod_{i=1}^{\ell}|A_i|}{q^{\ell}}+q^{d}|A_{\ell+1}|\prod_{i=1}^{\ell}|A_i|\left(\sum_{j=1}^{\ell}\frac{1}{|A_j|}\right)\right),\\
  \end{split}
\end{equation}
where we use Lemma \ref{yinli} in the last inequality. Put $C=\ell$ so that
$$
\frac{|A_{\ell+1}|^2\prod_{i=1}^{\ell}|A_i|}{\ell q^{\ell}}\geq q^{d}|A_{\ell+1}|\prod_{i=1}^{\ell}|A_i|\left(\frac{1}{|A_j|}\right)
$$
for every $1\leq j\leq\ell$. Thus
\begin{equation}
  \begin{split}
|S|^2\leq&\left(\sum_{(x_1, x_2, \ldots, x_{\ell})\in S'}\left|\Pi_{x_1, x_2, \ldots, x_\ell}\left(S(x_1, x_2, \ldots, x_{\ell})\right)\right|\right)\cdot\\
&\left(\frac{|A_{\ell+1}|^2\prod_{i=1}^{\ell}|A_i|}{q^{\ell}}+q^{d}|A_{\ell+1}|\prod_{i=1}^{\ell}|A_i|\left(\sum_{j=1}^{\ell}\frac{1}{|A_j|}\right)\right)\\
\leq&\left(\sum_{(x_1, x_2, \ldots, x_{\ell})\in S'}\left|\Pi_{x_1, x_2, \ldots, x_\ell}\left(S(x_1, x_2, \ldots, x_{\ell})\right)\right|\right)\cdot2\left(\frac{|A_{\ell+1}|^2\prod_{i=1}^{\ell}|A_i|}{q^{\ell}}\right).\\
  \end{split}
\end{equation}
Since $|S|\sim\prod_{i=1}^{\ell+1}|A_i|$, it follows that $S'\sim\prod_{i=1}^{\ell}|A_i|$. We conclude that
$$
\frac{1}{|S'|}\sum_{(x_1, x_2, \ldots, x_\ell)\in S'}\left|\Pi_{x_1, x_2, \ldots, x_\ell}\left(S(x_1, x_2, \ldots, x_\ell)\right)\right|\geq\frac{|S|^2}{|S'|\cdot2\left(\frac{|A_{\ell+1}|^2\prod_{i=1}^{\ell}|A_i|}{q^{\ell}}\right)}\sim q^\ell.
$$

\end{proof}
The following corollary is immediate.
\begin{corollary}\label{tuilun}
For integers $d$ and $\ell$ with $2\leq \ell\leq d$, there exists $C=C(\ell)$ such that the following holds. Given $A_1, A_2, \ldots, A_{\ell}, A_{\ell+1}\subseteq\mathbb{F}_q^d$, let $S\subseteq\prod_{i=1}^{\ell+1}A_i$ with $|S|\sim\prod_{i=1}^{\ell+1}|A_i|$ and $S'$ as in Theorem \ref{dingli}. If $|A_i||A_{\ell+1}|\geq Cq^{d+\ell}$ for $1\leq i\leq\ell$, then there exists $S_1\subseteq S'\subseteq\prod_{i=1}^{\ell}A_i$ with $|S_1|\sim|S'|\sim\prod_{i=1}^{\ell}|A_i|$ such that
$$
\left|\Pi_{x_1, x_2, \ldots, x_\ell}\left(S(x_1, x_2, \ldots, x_\ell)\right)\right|\gtrsim q^\ell,
$$
for every $(x_1, x_2, \ldots, x_\ell)\in S_1$.
\end{corollary}

We are ready to prove Theorem \ref{tuxin}.
\begin{proof}[Proof of Theorem \ref{tuxin}]
We use induction on $k$.

When $k=2$, $G$ has two vertices and one edge. Applying Corollary \ref{tuilun} with $\ell=1$ yields that there exists $S_1\subseteq S'\subseteq A_1$ with $|S_1|\sim|S'|\sim|A_1|$ such that
$$
\left|\Pi_{x_1}\left(S(x_1)\right)\right|\gtrsim q,
$$
for every $x_1\in S_1$. Recall that $\Pi_{x_1}\left(S(x_1)\right)=\{x_1\cdot y\in \mathbb{F}_q:y\in S(x_1)\}$ and $\Pi_{G}(S)=\{x\cdot y\in \mathbb{F}_q:(x, y)\in S\}$. So $\Pi_{x_1}\left(S(x_1)\right)\subseteq\Pi_{G}(S)$, and hence $\left|\Pi_{G}(S)\right|\gtrsim \left|\Pi_{x_1}\left(S(x_1)\right)\right|\gtrsim q.$

Now suppose the conclusion holds for k and consider the case $k_+1$. Since $G$ is connected, there exists a vertex such that after removing this vertex the resulting graph is still connected. Without loss generality, we can assume that by removing $v_{k+1}$, the resulting graph $G_1:=G[v_1, v_2, \ldots, v_{k}]$ is connected. Let $E_1$ be the edge set of $G_1$. Moreover, assume that the neighbor of $v_{i+1}$ in $G_i$ is $\{v_1, v_2, \ldots, v_\ell\}$ and hence $E=E_1\bigcup\left(\bigcup_{i=1}^\ell\{v_i, v_{k+1}\}\right)$.

Let $A_1, A_2, \ldots, A_{k+1}\subseteq\mathbb{F}_q^d$ satisfy $|A_i||A_{j}|\geq Cq^{d+k}$ for every $\{v_i, v_j\}\in E$ and $S\subseteq\prod_{i=1}^{k+1}A_i$ with $|S|\sim\prod_{i=1}^{k+1}|A_i|$. Define
$$
\tilde{S}=\{(x_1, x_2, \ldots, x_{\ell}, x_{k+1})\in \prod_{i=1}^{\ell}A_i\times A_{k+1}:(x_1, x_2, \ldots, x_\ell, x_{\ell+1}, \ldots, x_k, x_{k+1})\in S \text{ for some } (x_{\ell+1}, \ldots, x_{k})\}.
$$
Then $|\tilde{S}|\sim\prod_{i=1}^{\ell}|A_i|\cdot |A_{k+1}|$. Moreover, let
$$
\tilde{S}'=\{(x_1, x_2, \ldots, x_{\ell})\in \prod_{i=1}^{\ell}A_i:(x_1, x_2, \ldots, x_\ell, x_{k+1})\in \tilde{S} \text{ for some } x_{k+1}\}
$$
and
$$
\tilde{S}(x_1, x_2, \ldots, x_\ell)=\{x_{k+1}\in A_{k+1}: (x_1, x_2, \ldots, x_\ell, x_{k+1})\in \tilde{S}\}.
$$
Since $|A_i||A_{k+1}|\geq Cq^{d+k}\geq Cq^{d+\ell}$ for $1\leq i\leq\ell$, applying Corollary \ref{tuilun} with $\tilde{S}$ and $\tilde{S}'$ yields that there exists $\tilde{S}_1\subseteq \tilde{S}'\subseteq\prod_{i=1}^{\ell}A_i$ with $|\tilde{S}_1|\sim|\tilde{S}'|\sim\prod_{i=1}^\ell|A_i|$, such that for every $(x_1, x_2, \ldots, x_\ell)\in \tilde{S}_1$, we have
\begin{equation}\label{houmian}
\Pi_{x_1, x_2, \ldots, x_\ell}(\tilde{S}(x_1, x_2, \ldots, x_\ell))\gtrsim q^\ell.
\end{equation}
Now let
$$
S_2=\{(x_1, x_2, \ldots, x_{k}): (x_1, x_2, \ldots, x_{\ell})\in \tilde{S}_1 \text{ and } (x_1, x_2, \ldots, x_{k}, x_{k+1})\in S\}.
$$
So $|S_2|\sim \prod_{i=1}^k|A_i|$ and by the inductive hypothesis
\begin{equation}\label{qianmian}
\left|\Pi_{G_1}(S_2)\right|\gtrsim q^{|E_1|}.
\end{equation}
Combining inequalities (\ref{houmian}) and (\ref{qianmian}), the proof is completed.
\end{proof}

\section{Proof of Theorem \ref{shuxin}}
In this section, we prove Theorem \ref{shuxin}, where the graph is a tree.

For $A\subseteq\mathbb{F}_q^d$, $x\in\mathbb{F}_q^d$ and $\alpha\in\mathbb{F}_q^*$, let
$$
S_{A, \alpha}(x)=\sum_{s\in\mathbb{F}_q^*}\sum_{y\in A}\chi(s(x\cdot y-\alpha)).
$$
Let $B\subseteq\mathbb{F}_q^d$ be another set, not necessarily equal to $A$. Before proving Theorem \ref{shuxin}, we have the following proposition.
\begin{proposition}\label{main}
For $\alpha\neq0$, we have
$$
\left|\sum_{x\in B}S_{A, \alpha}(x)\right|\leq(|A||B|)^{1/2}q^{(d+1)/2}.
$$
\end{proposition}
\begin{proof}
Given $A, B\subseteq\mathbb{F}_q^d$, we denote $\sum_{x\in B}S_{A, \alpha}(x)$ by $v(\alpha)$ for short. Using Cauchy-Schwarz inequality, we have
\begin{equation}
  \begin{split}
   |v(\alpha)|^2 & =\left|\sum_{x\in B}\sum_{s\in\mathbb{F}_q^*}\sum_{y\in A}\chi(s(x\cdot y-\alpha))\right|^2\\
      & \leq\left|\sum_{x\in B}1^2\right|\left(\sum_{x\in B}\left|\sum_{s\in\mathbb{F}_q^*}\sum_{y\in A}\chi(s(x\cdot y-\alpha))\right|^2\right)\\
      & \leq\left|\sum_{x\in B}1^2\right|\left(\sum_{x\in \mathbb{F}_q^d}\left|\sum_{s\in\mathbb{F}_q^*}\sum_{y\in A}\chi(s(x\cdot y-\alpha))\right|^2\right)\\
      &=|B|\sum_{x\in \mathbb{F}_q^d}\sum_{s, s'\in\mathbb{F}_q^*}\sum_{y, y'\in A}\chi(s(x\cdot y-\alpha))\chi(s'(\alpha-x\cdot y'))\\
      &=|B|\sum_{x\in \mathbb{F}_q^d}\sum_{s, s'\in\mathbb{F}_q^*}\sum_{y, y'\in A}\chi(x\cdot(sy-s'y'))\chi(\alpha(s'-s)).
  \end{split}
\end{equation}
Note that
$$
\sum_{x\in \mathbb{F}_q^d}\chi(x\cdot(sy-s'y'))=
\begin{cases}
0, & \text{if }sy-s'y'\neq0;\\
q^d,& \text{if }sy-s'y'=0.
\end{cases}
$$
So
\begin{equation}\label{11}
  \begin{split}
   |v(\alpha)|^2 &\leq|B|\sum_{x\in \mathbb{F}_q^d}\sum_{s, s'\in\mathbb{F}_q^*}\sum_{y, y'\in A}\chi(x\cdot(sy-s'y'))\chi(\alpha(s'-s))\\
   &=|B|q^d\sum_{s, s'\in\mathbb{F}_q^*}\sum_{\substack{y, y'\in A\\sy=s'y'}}\chi(\alpha(s'-s)).
  \end{split}
\end{equation}
A result from \cite{MR2775806} shows that
\begin{equation}\label{21}
\left|\sum_{s, s'\in\mathbb{F}_q^*}\sum_{\substack{y, y'\in A\\sy=s'y'}}\chi(\alpha(s'-s))\right|\leq|A|q.
\end{equation}
Combining equations (\ref{11}) and (\ref{21}), we conclude that
$$
\left|\sum_{x\in B}S_{A, \alpha}(x)\right|=|v(\alpha)|\leq(|A||B|)^{1/2}q^{(d+1)/2}.
$$
\end{proof}

First of all, we consider the case that $k=2$. In this case, $G$ is a path with two edges. Without loss of generality, assume that $E=\{\{v_1, v_2\}, \{v_2, v_3\}\}$. For $\alpha_1, \alpha_2\in\mathbb{F}_q^*$, define
$$
\Pi_{\alpha_1, \alpha_2, G}(A_1\times A_2\times A_{3})=\{(x_1, x_2, x_3)\in A_1\times A_2\times A_3:\alpha_1=x_1\cdot x_2, \alpha_2=x_2\cdot x_3\}.
$$
Then
\begin{equation}
  \begin{split}
   |\Pi_{\alpha_1, \alpha_2, G}(A_1\times A_2\times A_{3})| & =|\{(x_1, x_2, x_3)\in A_1\times A_2\times A_3:\alpha_1=x_1\cdot x_2, \alpha_2=x_2\cdot x_3\}|\\
      &=q^{-2}\sum_{s, t\in\mathbb{F}_q}\sum_{x_i\in A_i}\chi(s(x_1\cdot x_2-\alpha_1))\chi(t(x_2\cdot x_3-\alpha_2))\\
      &=I+II+III,
  \end{split}
\end{equation}
where $I$ is the case that $s=t=0$, $II$ is the case that exactly one of $s$ and $t$ is nonzero, and $III$ is the case that $s\neq0$ and $t\neq0$.
\begin{equation}
  \begin{split}
    I & = q^{-2}\sum_{s=t=0}\sum_{x_i\in A_i}\chi(s(x_1\cdot x_2-\alpha_1))\chi(t(x_2\cdot x_3-\alpha_2))\\
      & =q^{-2}\sum_{x_i\in A_i}1\\
      &=q^{-2}|A_1||A_2||A_3|.
  \end{split}
\end{equation}
\begin{equation}
  \begin{split}
    II & = q^{-2}\left(\sum_{\substack{s=0\\t\neq0}}\sum_{x_i\in A_i}\chi(s(x_1\cdot x_2-\alpha_1))\chi(t(x_2\cdot x_3-\alpha_2))+\sum_{\substack{s\neq0\\t=0}}\sum_{x_i\in A_i}\chi(s(x_1\cdot x_2-\alpha_1))\chi(t(x_2\cdot x_3-\alpha_2))\right)\\
      & =q^{-2}\left(\sum_{t\neq0}\sum_{x_i\in A_i}\chi(t(x_2\cdot x_3-\alpha_2))+\sum_{s\neq0}\sum_{x_i\in A_i}\chi(s(x_1\cdot x_2-\alpha_1))\right)\\
      &=q^{-2}\left(|A_1|\sum_{t\neq0}\sum_{\substack{x_2\in A_2\\x_3\in A_3}}\chi(t(x_2\cdot x_3-\alpha_2))+|A_3|\sum_{s\neq0}\sum_{\substack{x_1\in A_1\\x_2\in A_2}}\chi(s(\alpha_1-x_2\cdot x_1))\right)\\
      &=q^{-2}\left(|A_1|\sum_{x\in A_2}S_{A_3, \alpha_2}(x)+|A_3|\sum_{x\in A_2}S_{A_1, \alpha_1}(x)\right).
  \end{split}
\end{equation}
Applying Proposition \ref{main}, we have
\begin{equation}
  \begin{split}
    |II| & \leq q^{-2}\left(|A_1|\left|\sum_{x\in A_2}S_{A_3, \alpha_2}(x)\right|+|A_3|\left|\sum_{x\in A_2}S_{A_1, \alpha_1}(x)\right|\right) \\
      & \leq|A_1|(|A_2||A_3|)^{1/2}q^{(d-3)/2}+|A_3|(|A_2||A_1|)^{1/2}q^{(d-3)/2}.
  \end{split}
\end{equation}
\begin{equation}
  \begin{split}
    |III| & = q^{-2}\left|\sum_{\substack{s\neq0\\t\neq0}}\sum_{x_i\in A_i}\chi(s(x_1\cdot x_2-\alpha_1))\chi(t(x_2\cdot x_3-\alpha_2))\right|\\
      & \leq q^{-2}\sum_{x_2\in A_2}\left|\sum_{s\neq0}\sum_{x_1\in A_1}\chi(s(x_1\cdot x_2-\alpha_1))\right|\left|\sum_{t\neq0}\sum_{x_3\in A_3}\chi(t(x_2\cdot x_3-\alpha_2))\right|\\
      &\leq q^{-2}\sum_{x_2\in \mathbb{F}_q^d}\left|\sum_{s\neq0}\sum_{x_1\in A_1}\chi(s(x_1\cdot x_2-\alpha_1))\right|\left|\sum_{t\neq0}\sum_{x_3\in A_3}\chi(t(x_2\cdot x_3-\alpha_2))\right|\\
      &\leq q^{-2}\left(\sum_{x_2\in \mathbb{F}_q^d}\left|\sum_{s\neq0}\sum_{x_1\in A_1}\chi(s(x_1\cdot x_2-\alpha_1))\right|^2\right)^{1/2}\left(\sum_{x_2\in \mathbb{F}_q^d}\left|\sum_{t\neq0}\sum_{x_3\in A_3}\chi(t(x_2\cdot x_3-\alpha_2))\right|^2\right)^{1/2}.
  \end{split}
\end{equation}
In the proof of Proposition \ref{main}, we know that
$$
\left(\sum_{x_2\in \mathbb{F}_q^d}\left|\sum_{s\neq0}\sum_{x_1\in A}\chi(s(x_1\cdot x_2-\alpha_1))\right|^2\right)^{1/2}\leq|A|^{1/2}q^{(d+1)/2}.
$$
Thus
\begin{equation}
  \begin{split}
    |III|  &\leq q^{-2}\left(\sum_{x_2\in \mathbb{F}_q^d}\left|\sum_{s\neq0}\sum_{x_1\in A_1}\chi(s(x_1\cdot x_2-\alpha_1))\right|^2\right)^{1/2}\left(\sum_{x_2\in \mathbb{F}_q^d}\left|\sum_{t\neq0}\sum_{x_3\in A_3}\chi(t(x_2\cdot x_3-\alpha_2))\right|^2\right)^{1/2}\\
    &\leq q^{d-1}|A_1|^{1/2}|A_3|^{1/2}.
  \end{split}
\end{equation}
We can choose $C=100$. If $|A_1||A_2|\geq100q^{d+1}$ and $|A_2||A_3|\geq100q^{d+1}$, then
$$
\left\{
\begin{array}{c}
  q^{-2}|A_1||A_2||A_3|\geq4|A_1|(|A_2||A_3|)^{1/2}q^{(d-3)/2},  \\
  q^{-2}|A_1||A_2||A_3|\geq4|A_3|(|A_2||A_1|)^{1/2}q^{(d-3)/2}, \\
  q^{-2}|A_1||A_2||A_3|\geq4(|A_1||A_3|)^{1/2}q^{d-1}.
\end{array}
\right.
$$
So
$$
|\Pi_{\alpha_1, \alpha_2, G}(A_1\times A_2\times A_{3})|=I+II+III\geq I-|II|-|III|\geq\frac{1}{4}q^{-2}|A_1||A_2||A_3|>0.
$$
Since $|\Pi_{\alpha_1, \alpha_2, G}(A_1\times A_2\times A_{3})|$ is always a nonnegative integer, it follows that $|\Pi_{\alpha_1, \alpha_2, G}(A_1\times A_2\times A_{3})|\geq1$, and hence $(\alpha_1, \alpha_2)\in\Pi_G(A_1\times A_2\times A_{3})$. Thus $\Pi_G(A_1\times A_2\times A_{3})\supseteq(\mathbb{F}_q^*)^2$.
\\

Next we consider the case that $k\geq3$. For a function $\alpha:E\rightarrow\mathbb{F}_q^*$, we write $\alpha_{\{i, j\}}=\alpha(\{v_i, v_j\})$ for convenience and define
$$
\begin{array}{lll}
&\Pi_{\vec{\alpha}, G}(A_1\times A_2\times \cdots\times A_{k+1})\\
=&\{(x_1, x_2,\ldots, x_{k+1})\in A_1\times A_2\times\cdots\times A_{k+1}:\alpha_{\{i, j\}}=x_i\cdot x_{j}, \text{ for every }\{v_i, v_j\}\in E\}.
\end{array}
$$
Then
\begin{equation}
  \begin{split}
   &|\Pi_{\vec{\alpha}, G}(A_1\times A_2\times \cdots\times A_{k+1})| \\
    =&|\{(x_1, x_2,\ldots, x_{k+1})\in A_1\times A_2\times\cdots\times A_{k+1}:\alpha_{\{i, j\}}=x_i\cdot x_{j}, \text{ for every }\{v_i, v_j\}\in E\}|\\
      =&q^{-k}\sum_{\vec{s}\in\mathbb{F}_q^k}\sum_{x_i\in A_i}\prod_{\{v_i,v_j\}\in E}\chi(s_{\{i, j\}}(x_i\cdot x_{j}-\alpha_{\{i, j\}}))\\
      =&R_0+R_1+\cdots+R_k,
  \end{split}
\end{equation}
where $\vec{s}\in\mathbb{F}_q^k$ indexed by elements in $E$, i.e. $s_{\{i, j\}}\in\mathbb{F}_q$ is the $\{v_i, v_j\}$'th coordinate of $\vec{s}$, and  $R_m(0\leq m\leq k)$ is the case that exactly $m$ coordinates of $\vec{s}$ are nonzero.
\begin{equation}
  \begin{split}
    R_0 & = q^{-k}\sum_{\vec{s}=0}\sum_{x_i\in A_i}\prod_{\{v_i, v_j\}\in E}\chi(s_{\{i, j\}}(x_i\cdot x_{j}-\alpha_{\{i, j\}}))\\
      & =q^{-k}\sum_{x_i\in A_i}1\\
      &=q^{-k}\prod_{i=1}^{k+1}|A_i|.
  \end{split}
\end{equation}
\begin{equation}
  \begin{split}
    R_1 & = q^{-k}\sum_{\{u, v\}\in E}\sum_{\substack{\vec{s}\in\mathbb{F}_q^k\\\text{only }s_{\{u, v\}}\text{ is nonzero}}}\sum_{x_i\in A_i}\prod_{\{v_i, v_j\}\in E}\chi(s_{\{i, j\}}(x_i\cdot x_{j}-\alpha_{\{i, j\}}))\\
      & =q^{-k}\sum_{\{u, v\}\in E}\sum_{s_{\{u, v\}}\in\mathbb{F}_q^*}\sum_{x_i\in A_i}\chi(s_{\{u, v\}}(x_{u}\cdot x_{v}-\alpha_{\{u, v\}}))\\
      & =q^{-k}\sum_{\{u, v\}\in E}\prod_{\substack{1\leq i\leq k+1\\i\neq u, v}}|A_i|\left(\sum_{s_{\{u, v\}}\in\mathbb{F}_q^*}\sum_{\substack{x_u\in A_u\\x_{v}\in A_{v}}}\chi(s_{\{u, v\}}(x_u\cdot x_{v}-\alpha_{\{u, v\}}))\right)\\
      &=q^{-k}\sum_{\{u, v\}\in E}\left(\prod_{\substack{1\leq i\leq k+1\\i\neq u, v}}|A_i|\right)\sum_{x\in A_u}S_{A_{v}, \alpha_{\{u, v\}}}(x).
  \end{split}
\end{equation}
Applying Proposition \ref{main}, we have
\begin{equation}
  \begin{split}
    |R_1| &\leq q^{-k}\sum_{\{u, v\}\in E}\left(\prod_{\substack{1\leq i\leq k+1\\i\neq u, v}}|A_i|\right)\left|\sum_{x\in A_u}S_{A_{v}, \alpha_{\{u, v\}}}(x)\right|\\
    &\leq q^{-k}\sum_{\{u, v\}\in E}\left(\prod_{\substack{1\leq i\leq k+1\\i\neq u, v}}|A_i|\right)(|A_u||A_{v}|)^{1/2}q^{(d+1)/2}\\
    &=q^{\frac{d+1}{2}-k}\sum_{\{u, v\}\in E}(|A_u||A_{v}|)^{1/2}\left(\prod_{\substack{1\leq i\leq k+1\\i\neq u, v}}|A_i|\right).
  \end{split}
\end{equation}

\begin{equation}
  \begin{split}
    &R_2 \\
    =&  q^{-k}\sum_{\substack{\{u_1, w_1\}, \{u_2, w_2\}\in E\\\{u_1, w_1\}\text{ and }\{u_2, w_2\}\text{ are different}}}\sum_{\substack{\vec{s}\in\mathbb{F}_q^k\\\text{only }s_{\{u_1, w_1\}}\text{ and }s_{\{u_2, w_2\}}\text{ are nonzero}}}\sum_{x_i\in A_i}\prod_{\{v_i, v_j\}\in E}\chi(s_{\{i, j\}}(x_i\cdot x_{j}-\alpha_{\{i, j\}}))\\
      =& q^{-k}\sum_{\substack{\{u_1, w_1\}, \{u_2, w_2\}\in E\\\{u_1, w_1\}\text{ and }\{u_2, w_2\}\text{ are different}\\\{u_1, w_1\}\text{ and }\{u_2, w_2\}\text{ have a common vertex}}}\sum_{\substack{\vec{s}\in\mathbb{F}_q^k\\\text{only }s_{\{u_1, w_1\}}\text{ and }s_{\{u_2, w_2\}}\text{ are nonzero}}}\sum_{x_i\in A_i}\prod_{\{v_i, v_j\}\in E}\chi(s_{\{i, j\}}(x_i\cdot x_{j}-\alpha_{\{i, j\}}))\\
      &+q^{-k}\sum_{\substack{\{u_1, w_1\}, \{u_2, w_2\}\in E\\\{u_1, w_1\}\text{ and }\{u_2, w_2\}\text{ are disjoint}}}\sum_{\substack{\vec{s}\in\mathbb{F}_q^k\\\text{only }s_{\{u_1, w_1\}}\text{ and }s_{\{u_2, w_2\}}\text{ are nonzero}}}\sum_{x_i\in A_i}\prod_{\{v_i, v_j\}\in E}\chi(s_{\{i, j\}}(x_i\cdot x_{j}-\alpha_{\{i, j\}})).
  \end{split}
\end{equation}
Given two edges $\{u_1, w_1\}, \{u_2, w_2\}\in E$ with a common vertex, without loss of generality assume that $u_1=u_2$. For simplicity, we write $u_1=u_2=u, s_{\{u_1, w_1\}}=s_1, s_{\{u_2, w_2\}}=s_2, \alpha_{\{u, w_1\}}=\alpha_1, \alpha_{\{u, w_2\}}=\alpha_2$. Then we have
\begin{equation}
  \begin{split}
    &  \sum_{\substack{\vec{s}\in\mathbb{F}_q^k\\\text{only }s_{\{u, w_1\}}\text{ and }s_{\{u, w_2\}}\text{ are nonzero}}}\sum_{x_i\in A_i}\prod_{\{v_i, v_j\}\in E}\chi(s_{\{i, j\}}(x_i\cdot x_{j}-\alpha_{\{i, j\}}))\\
      =&\sum_{s_{\{u, w_1\}}, s_{\{u, w_2\}}\in\mathbb{F}_q^*}\sum_{x_i\in A_i}\chi(s_{\{u, w_1\}}(x_{u}\cdot x_{w_1}-\alpha_{\{u, w_1\}}))\chi(s_{\{u, w_2\}}(x_{u}\cdot x_{w_2}-\alpha_{\{u, w_2\}}))\\
            =&\sum_{s_{1}, s_{2}\in\mathbb{F}_q^*}\sum_{x_i\in A_i}\chi(s_{1}(x_{u}\cdot x_{w_1}-\alpha_{1}))\chi(s_{2}(x_{u}\cdot x_{w_2}-\alpha_{2}))\\
      =&\left(\prod_{\substack{1\leq i\leq k+1\\i\neq u, w_1, w_2}}|A_i|\right)\left(\sum_{s_{1}, s_2\in\mathbb{F}_q^*}\sum_{\substack{x_{u}\in A_{u}\\x_{w_1}\in A_{w_1}\\x_{w_2}\in A_{w_2}}}\chi(s_{1}(x_{u}\cdot x_{w_1}-\alpha_{1}))\chi(s_{2}(x_{u}\cdot x_{w_2}-\alpha_{2}))\right).
  \end{split}
\end{equation}
\begin{equation}
\begin{split}
&\left|\sum_{s_{1}, s_2\in\mathbb{F}_q^*}\sum_{\substack{x_{u}\in A_{u}\\x_{w_1}\in A_{w_1}\\x_{w_2}\in A_{w_2}}}\chi(s_{1}(x_{u}\cdot x_{w_1}-\alpha_{1}))\chi(s_{2}(x_{u}\cdot x_{w_2}-\alpha_{2}))\right|\\
\leq&\sum_{x_{u}\in A_{u}}\left(\left|\sum_{s_{1}\in\mathbb{F}_q^*}\sum_{x_{w_1}\in A_{w_1}}\chi(s_{1}(x_{u}\cdot x_{w_1}-\alpha_{1}))\right|\left|\sum_{s_{2}\in\mathbb{F}_q^*}\sum_{x_{w_2}\in A_{w_2}}\chi(s_{2}(x_{u}\cdot x_{w_2}-\alpha_{2}))\right|\right)\\
\leq&\sum_{x_{u}\in \mathbb{F}_q^d}\left(\left|\sum_{s_{1}\in\mathbb{F}_q^*}\sum_{x_{w_1}\in A_{w_1}}\chi(s_{1}(x_{u}\cdot x_{w_1}-\alpha_{1}))\right|\left|\sum_{s_{2}\in\mathbb{F}_q^*}\sum_{x_{w_2}\in A_{w_2}}\chi(s_{2}(x_{u}\cdot x_{w_2}-\alpha_{2}))\right|\right)\\
\leq&\left(\sum_{x_{u}\in \mathbb{F}_q^d}\left|\sum_{s_{1}\in\mathbb{F}_q^*}\sum_{x_{w_1}\in A_{w_1}}\chi(s_{1}(x_u\cdot x_{w_1}-\alpha_{1}))\right|^2\right)^{\frac12}\left(\sum_{x_{u}\in \mathbb{F}_q^d}\left|\sum_{s_{2}\in\mathbb{F}_q^*}\sum_{x_{w_2}\in A_{w_2}}\chi(s_{2}(x_{u}\cdot x_{w_2}-\alpha_{2}))\right|^2\right)^{\frac12}.
\end{split}
\end{equation}
Similar to inequality (\ref{21}), we have
$$
\sum_{x_{u}\in \mathbb{F}_q^d}\left|\sum_{s_{1}\in\mathbb{F}_q^*}\sum_{x_{w_1}\in A_{w_1}}\chi(s_{1}(x_u\cdot x_{w_1}-\alpha_{1}))\right|^2\leq q^{d+1}|A_{w_1}|,
$$
and
$$
\sum_{x_{u}\in \mathbb{F}_q^d}\left|\sum_{s_{2}\in\mathbb{F}_q^*}\sum_{x_{w_2}\in A_{w_2}}\chi(s_{2}(x_{u}\cdot x_{w_2}-\alpha_{2}))\right|^2\leq q^{d+1}|A_{w_2}|.
$$
So we have
\begin{equation}
\begin{split}
&\left|\sum_{s_{1}, s_2\in\mathbb{F}_q^*}\sum_{\substack{x_{u}\in A_{u}\\x_{w_1}\in A_{w_1}\\x_{w_2}\in A_{w_2}}}\chi(s_{1}(x_{u}\cdot x_{w_1}-\alpha_{1}))\chi(s_{2}(x_{u}\cdot x_{w_2}-\alpha_{2}))\right|\\
\leq&\left(\sum_{x_{u}\in \mathbb{F}_q^d}\left|\sum_{s_{1}\in\mathbb{F}_q^*}\sum_{x_{w_1}\in A_{w_1}}\chi(s_{1}(x_u\cdot x_{w_1}-\alpha_{1}))\right|^2\right)^{\frac12}\left(\sum_{x_{u}\in \mathbb{F}_q^d}\left|\sum_{s_{2}\in\mathbb{F}_q^*}\sum_{x_{w_2}\in A_{w_2}}\chi(s_{2}(x_{u}\cdot x_{w_2}-\alpha_{2}))\right|^2\right)^{\frac12}\\
\leq&q^{d+1}(|A_{w_1}||A_{w_2}|)^{\frac12}.
\end{split}
\end{equation}

Given two disjoint edges $\{u_1, w_1\}, \{u_2, w_2\}\in E$, for simplicity, we write $s_{\{u_1, w_1\}}=s_1, s_{\{u_2, w_2\}}=s_2, \alpha_{\{u_1, w_1\}}=\alpha_1, \alpha_{\{u_2, w_2\}}=\alpha_2$. Then we have
\begin{equation}
  \begin{split}
 &  \sum_{\substack{\vec{s}\in\mathbb{F}_q^k\\\text{only }s_{\{u_1, w_1\}}\text{ and }s_{\{u_2, w_2\}}\text{ are nonzero}}}\sum_{x_i\in A_i}\prod_{\{v_i, v_j\}\in E}\chi(s_{\{i, j\}}(x_i\cdot x_{j}-\alpha_{\{i, j\}}))\\
      =&\sum_{s_{\{u_1, w_1\}}, s_{\{u_2, w_2\}}\in\mathbb{F}_q^*}\sum_{x_i\in A_i}\chi(s_{\{u_1, w_1\}}(x_{u_1}\cdot x_{w_1}-\alpha_{\{u_1, w_1\}}))\chi(s_{\{u_2, w_2\}}(x_{u_2}\cdot x_{w_2}-\alpha_{\{u_2, w_2\}}))\\
            =&\sum_{s_{1}, s_{2}\in\mathbb{F}_q^*}\sum_{x_i\in A_i}\chi(s_{1}(x_{u_1}\cdot x_{w_1}-\alpha_{1}))\chi(s_{2}(x_{u_2}\cdot x_{w_2}-\alpha_{2}))\\
      =&\left(\prod_{\substack{1\leq i\leq k+1\\i\neq u_1, w_1, u_2, w_2}}|A_i|\right)\left(\sum_{s_{1}\in\mathbb{F}_q^*}\sum_{\substack{x_{u_1}\in A_{u_1}\\x_{w_1}\in A_{w_1}}}\chi(s_{1}(x_{u_1}\cdot x_{w_1}-\alpha_{1}))\right)\left(\sum_{s_{2}\in\mathbb{F}_q^*}\sum_{\substack{x_{u_2}\in A_{u_2}\\x_{w_2}\in A_{w_2}}}\chi(s_{2}(x_{u_2}\cdot x_{w_2}-\alpha_{2}))\right)\\
      =&\left(\prod_{\substack{1\leq i\leq k+1\\i\neq u_1, w_1, u_2, w_2}}|A_i|\right)\left(\sum_{x\in A_{u_1}}S_{A_{w_1}, \alpha_{1}}(x)\right)\left(\sum_{x\in A_{u_2}}S_{A_{w_2}, \alpha_{2}}(x)\right).
  \end{split}
\end{equation}
Applying Proposition \ref{main}, we have
\begin{equation}
  \begin{split}
&\left|\sum_{\substack{\vec{s}\in\mathbb{F}_q^k\\\text{only }s_{\{u_1, w_1\}}\text{ and }s_{\{u_2, w_2\}}\text{ are nonzero}}}\sum_{x_i\in A_i}\prod_{\{v_i, v_j\}\in E}\chi(s_{\{i, j\}}(x_i\cdot x_{j}-\alpha_{\{i, j\}}))\right|\\
=&\left|\prod_{\substack{1\leq i\leq k+1\\i\neq u_1, w_1, u_2, w_2}}|A_i|\right|\left|\sum_{x\in A_{u_1}}S_{A_{w_1}, \alpha_{1}}(x)\right|\left|\sum_{x\in A_{u_2}}S_{A_{w_2}, \alpha_{2}}(x)\right|\\
    \leq&q^{d+1}\left(\prod_{\substack{1\leq i\leq k+1\\i\neq u_1, w_1, u_2, w_2}}|A_i|\right)(|A_{u_1}||A_{w_1}||A_{u_2}||A_{w_2}|)^{1/2}.
  \end{split}
\end{equation}
Therefore,
\begin{equation}
  \begin{split}
      &|R_2| \\
  \leq&  q^{-k}\left|\sum_{\substack{\{u_1, w_1\}, \{u_2, w_2\}\in E\\\{u_1, w_1\}\text{ and }\{u_2, w_2\}\text{ are different}\\\{u_1, w_1\}\text{ and }\{u_2, w_2\}\text{ have a common vertex}}}\sum_{\substack{\vec{s}\in\mathbb{F}_q^k\\\text{only }s_{\{u_1, w_1\}}\text{ and }s_{\{u_2, w_2\}}\text{ are nonzero}}}\sum_{x_i\in A_i}\prod_{\{v_i, v_j\}\in E}\chi(s_{\{i, j\}}(x_i\cdot x_{j}-\alpha_{\{i, j\}}))\right|\\
      &+q^{-k}\left|\sum_{\substack{\{u_1, w_1\}, \{u_2, w_2\}\in E\\\{u_1, w_1\}\text{ and }\{u_2, w_2\}\text{ are disjoint}}}\sum_{\substack{\vec{s}\in\mathbb{F}_q^k\\\text{only }s_{\{u_1, w_1\}}\text{ and }s_{\{u_2, w_2\}}\text{ are nonzero}}}\sum_{x_i\in A_i}\prod_{\{v_i, v_j\}\in E}\chi(s_{\{i, j\}}(x_i\cdot x_{j}-\alpha_{\{i, j\}}))\right|\\
  \leq&q^{d+1-k}\sum_{\substack{\{u, w_1\}, \{u, w_2\}\in E\\w_1\neq w_2}}\left(\prod_{\substack{1\leq i\leq k+1\\i\neq u, w_1, w_2}}|A_i|\right)(|A_{w_1}||A_{w_2}|)^{\frac12}\\
      &+q^{d+1-k}\sum_{\substack{\{u_1, w_1\}, \{u_2, w_2\}\in E\\\{u_1, w_1\}\text{ and }\{u_2, w_2\}\text{ are disjoint}}}\left(\prod_{\substack{1\leq i\leq k+1\\i\neq u_1, w_1, u_2, w_2}}|A_i|\right)(|A_{u_1}||A_{w_1}||A_{u_2}||A_{w_2}|)^{1/2}.
  \end{split}
\end{equation}

Next we consider $R_m$ for $m\geq3$. We first choose $m$ edges from $E$, $\{u_1, w_1\}, \{u_2, w_2\}, \ldots, \{u_m, w_m\}$, and consider
$$
q^{-k}\sum_{\substack{\vec{s}\in\mathbb{F}_q^k\\\text{exactly }s_{\{u_i, w_i\}}(1\leq i\leq m)\text{ are nonzero}}}\sum_{x_i\in A_i}\prod_{\{v_i, v_j\}\in E}\chi(s_{\{i, j\}}(x_i\cdot x_{j}-\alpha_{\{i, j\}})).
$$
For convenience, let $S=\{u_1, u_2, \ldots, u_m, w_1, w_2, \ldots, w_m\}$ (some of these vertices may be the same, so $|S|\leq 2m$), $s_1=s_{\{u_1, w_1\}}, s_2=s_{\{u_2, w_2\}}, \ldots, s_m=s_{\{u_m, w_m\}}$, and $\alpha_1=\alpha_{\{u_1, w_1\}}, \alpha_2=\alpha_{\{u_2, w_2\}}, \ldots, \alpha_m=\alpha_{\{u_m, w_m\}}$. Roughly speaking, the vertices in $S$ are those we are concerned with. Now
\begin{equation}
  \begin{split}
   &q^{-k}\sum_{\substack{\vec{s}\in\mathbb{F}_q^k\\\text{exactly }s_{\{u_i, w_i\}}(1\leq i\leq m)\text{ are nonzero}}}\sum_{x_i\in A_i}\prod_{\{v_i, v_j\}\in E}\chi(s_{\{i, j\}}(x_i\cdot x_{j}-\alpha_{\{i, j\}}))\\
  =&q^{-k}\sum_{s_1, s_2, \ldots, s_m\in\mathbb{F}_q^*}\sum_{x_i\in A_i}\prod_{\{v_i, v_j\}\in E}\chi(s_{\{i, j\}}(x_i\cdot x_{j}-\alpha_{\{i, j\}}))\\
  =&q^{-k}\sum_{s_1, s_2, \ldots, s_m\in\mathbb{F}_q^*}\sum_{x_i\in A_i}\prod_{\ell=1}^m\chi(s_{\ell}(x_{u_\ell}\cdot x_{w_\ell}-\alpha_{\ell})).
  \end{split}
\end{equation}

For a vertex $u_i$ (or $w_i$) in $S$, we say $u_i$ (or $w_i$) is unique if $u_i\notin\{u_j, w_j\}$ whenever $j\neq i$. If a vertex $w_i$ is unique, then we exchange $u_i$ and $w_i$ so that $u_i$ is unique. So we only consider unique vertices of $u_i, 1\leq i\leq m$. Note that $G$ is a tree. There are at least $m$ edges in $G[S]$. So $|S|\geq m+1$. Hence there are at least two unique vertices in $S$, say $u_1$ and $u_m$. Moreover, assume that $w_1$ and $w_m$ are distinct.

We need the following lemma.
\begin{lemma}\label{quadratic}
Let $m$ and $n$ be positive integers. Then for each double sequence $\{c_{jk}:1\leq j\leq m, 1\leq k\leq n\}$ and pair of sequence $\{z_j:1\leq j\leq m\}$ and $\{y_k:1\leq k\leq n\}$ of complex numbers, we have the bound
$$
\left|\sum_{j=1}^m\sum_{k=1}^nc_{jk}z_jy_k\right|\leq\sqrt{RC}\left(\sum_{j=1}^m|z_j|^2\right)^{\frac12}\left(\sum_{k=1}^n|y_k|^2\right)^{\frac12},
$$
where $R$ and $C$ are respectively the row and column sum maxima defined by
$$
R=\max_j\sum_{k=1}^n|c_{jk}|\text{ and }C=\max_k\sum_{j=1}^n|c_{jk}|.
$$
\end{lemma}

We can write
$$
q^{-k}\sum_{s_1, s_2, \ldots, s_m\in\mathbb{F}_q^*}\sum_{x_i\in A_i}\prod_{\ell=1}^m\chi(s_{\ell}(x_{u_\ell}\cdot x_{w_\ell}-\alpha_{\ell}))=\sum_{x_{w_1}\in A_{w_1}}\sum_{x_{w_m}\in A_{w_m}}c_{x_{w_1}, x_{w_m}}z_{x_{w_1}}y_{x_{w_m}},
$$
where
$$
z_{x_{w_1}}=S_{A_{u_1}, \alpha_{1}}(x_{w_1}), \quad y_{x_{w_m}}=S_{A_{u_m}, \alpha_{m}}(x_{w_m}),
$$
and
$$c_{x_{w_1}, x_{w_m}}=q^{-k}\sum_{s_{2}, s_{3}, \ldots, s_{{m-1}}\in\mathbb{F}_q^*}\sum_{\substack{x_i\in A_i\\i\neq u_1, w_1, u_m, w_m}}\prod_{\ell=2}^{m-1}\chi(s_{\ell}(x_{u_\ell}\cdot x_{w_\ell}-\alpha_{\ell})).
$$
Note that
$$
R=\max_{x_{w_1}}\sum_{x_{w_m}\in A_{w_m}}|c_{x_{w_1}, x_{w_m}}|\leq q^{-k}\sum_{x_{w_m}\in A_{w_m}}\sum_{s_{2}, s_{3}, \ldots, s_{{m-1}}\in\mathbb{F}_q^*}\sum_{\substack{x_i\in A_i\\i\neq u_1, w_1, u_m, w_m}}1\leq q^{m-2-k}\prod_{\substack{1\leq i\leq k+1\\i\neq u_1, w_1, u_m}}|A_i|,
$$
and
$$
C=\max_{x_{w_m}}\sum_{x_{w_1}\in A_{w_1}}|c_{x_{w_1}, x_{w_m}}|\leq q^{-k}\sum_{x_{w_1}\in A_{w_1}}\sum_{s_{2}, s_{3}, \ldots, s_{{m-1}}\in\mathbb{F}_q^*}\sum_{\substack{x_i\in A_i\\i\neq u_1,w_1, u_m, w_m}}1\leq q^{m-2-k}\prod_{\substack{1\leq i\leq k+1\\i\neq u_1, u_m, w_m}}|A_i|,
$$
and applying Lemma \ref{quadratic} we have,
\begin{equation}
  \begin{split}
 &\left|q^{-k}\sum_{s_1, s_2, \ldots, s_m\in\mathbb{F}_q^*}\sum_{x_i\in A_i}\prod_{\ell=1}^m\chi(s_{\ell}(x_{u_\ell}\cdot x_{w_\ell}-\alpha_{\ell}))\right|\\
=&\left|\sum_{x_{w_1}\in A_{w_1}}\sum_{x_{w_m}\in A_{w_m}}c_{x_{w_1}, x_{w_m}}z_{x_{w_1}}y_{x_{w_m}}\right|\\
\leq&\sqrt{RC}\left(\sum_{x_{w_1}\in A_{w_1}}|z_{x_{w_1}}|^2\right)^{\frac12}\left(\sum_{x_{w_m}\in A_{w_m}}|y_{x_{w_m}}|^2\right)^{\frac12}\\
\leq&\sqrt{\left(q^{m-2-k}\prod_{\substack{1\leq i\leq k+1\\i\neq u_1, w_1, u_m}}|A_i|\right)\left(q^{m-2-k}\prod_{\substack{1\leq i\leq k+1\\i\neq u_1, u_m, w_m}}|A_i|\right)}\times\\
&\left(\sum_{x_{w_1}\in A_{w_1}}|S_{A_{u_1}, \alpha_{1}}(x_{w_1})|^2\right)^{\frac12}\left(\sum_{x_{w_m}\in A_{w_m}}|S_{A_{u_m}, \alpha_{m}}(x_{w_m})|^2\right)^{\frac12}\\
\leq&q^{m-2-k}\left(\prod_{\substack{1\leq i\leq k+1\\i\neq u_1, w_1,u_m, w_m}}|A_i|\right)|A_{v_1}|^{\frac12}|A_{w_m}|^{\frac12}\left(\sum_{x_{w_1}\in\mathbb{F}_q^d }|S_{A_{u_1}, \alpha_{1}}(x_{w_1})|^2\right)^{\frac12}\left(\sum_{x_{w_m}\in \mathbb{F}_q^d}|S_{A_{u_m}, \alpha_{m}}(x_{w_m})|^2\right)^{\frac12}\\
\leq&q^{d+m-1-k}\left(\prod_{\substack{1\leq i\leq k+1\\i\neq u_1, w_1,u_m, w_m}}|A_i|\right)(|A_{u_1}||A_{w_1}||A_{u_m}||A_{w_m}|)^{\frac12}.
  \end{split}
\end{equation}

If for every pair of unique vertices $u_i$ and $u_j$, we have $w_i=w_j$, then $G[S]$ is a star, i.e. $w_1=w_2=\cdots=w_m:=w$, and $S=\{u_1, u_2, \ldots, u_m, w\}$. In this case,
\begin{equation}
  \begin{split}
   &\left|q^{-k}\sum_{s_1, s_2, \ldots, s_m\in\mathbb{F}_q^*}\sum_{x_i\in A_i}\prod_{\ell=1}^m\chi(s_{\ell}(x_{u_\ell}\cdot x_{w_\ell}-\alpha_{\ell}))\right|\\
  =&\left|q^{-k}\sum_{s_1, s_2, \ldots, s_m\in\mathbb{F}_q^*}\sum_{x_i\in A_i}\prod_{\ell=1}^m\chi(s_{\ell}(x_{u_\ell}\cdot x_{w}-\alpha_{\ell}))\right|\\
  =&q^{-k}\left(\prod_{i\in V\backslash S}|A_i|\right)\left|\sum_{x_w\in A_w}\prod_{\ell=1}^m\left(\sum_{s_\ell\in\mathbb{F}_q^*}\sum_{x_{u_\ell}\in A_{u_\ell}}\chi(s_{\ell}(x_{u_\ell}\cdot x_{w}-\alpha_{\ell}))\right)\right|\\
  \leq&q^{-k}\left(\prod_{i\in V\backslash S}|A_i|\right)\sum_{x_w\in A_w}\prod_{\ell=1}^m\left|\left(\sum_{s_\ell\in\mathbb{F}_q^*}\sum_{x_{u_\ell}\in A_{u_\ell}}\chi(s_{\ell}(x_{u_\ell}\cdot x_{w}-\alpha_{\ell}))\right)\right|\\
    \leq&q^{-k}\left(\prod_{i\in V\backslash S}|A_i|\right)\sum_{x_w\in \mathbb{F}_q^d}\prod_{\ell=1}^m\left|\left(\sum_{s_\ell\in\mathbb{F}_q^*}\sum_{x_{u_\ell}\in A_{u_\ell}}\chi(s_{\ell}(x_{u_\ell}\cdot x_{w}-\alpha_{l}))\right)\right|\\
  \leq&q^{-k}\left(\prod_{i\in V\backslash S}|A_i|\right)\prod_{\ell=1}^m\left(\sum_{x_w\in \mathbb{F}_q^d}\left|\sum_{s_\ell\in\mathbb{F}_q^*}\sum_{x_{u_\ell}\in A_{u_\ell}}\chi(s_{\ell}(x_{u_\ell}\cdot x_{w}-\alpha_{\ell}))\right|^2\right)^{1/2},
  \end{split}
\end{equation}
where the last inequality follows from
$$
\sum_{i=1}^n\prod_{j=1}^m|a_{ij}|\leq\prod_{j=1}^m\left(\sum_{i=1}^n|a_{ij}|^2\right)^{1/2},
$$
a generalization of Cauchy-Schwarz inequality. Since
$$
\left(\sum_{x_w\in \mathbb{F}_q^d}\left|\sum_{s_\ell\in\mathbb{F}_q^*}\sum_{x_{u_\ell}\in A_{u_\ell}}\chi(s_{\ell}(x_{u_\ell}\cdot x_{w}-\alpha_{\ell}))\right|^2\right)^{1/2}\leq|A_{u_\ell}|^{1/2}q^{(d+1)/2},
$$
we have
\begin{equation}
  \begin{split}
   &\left|q^{-k}\sum_{s_1, s_2, \ldots, s_m\in\mathbb{F}_q^*}\sum_{x_i\in A_i}\prod_{\ell=1}^m\chi(s_{\ell}(x_{u_\ell}\cdot x_{w_\ell}-\alpha_{\ell}))\right|\\
\leq&q^{-k}\left(\prod_{i\in V\backslash S}|A_i|\right)\prod_{\ell=1}^m\left(\sum_{x_w\in \mathbb{F}_q^d}\left|\sum_{s_\ell\in\mathbb{F}_q^*}\sum_{x_{u_\ell}\in A_{u_\ell}}\chi(s_{\ell}(x_{u_\ell}\cdot x_{w}-\alpha_{\ell}))\right|^2\right)^{1/2},
  \end{split}
\end{equation}

In conclusion, we have
\begin{equation}
  \begin{split}
 &|R_1|+|R_2|+\cdots+|R_k|\\
 \leq &q^{\frac{d+1}{2}-k}\sum_{j=1}^k(|A_j||A_{j+1}|)^{1/2}\left(\prod_{\substack{1\leq i\leq k+1\\i\neq j, j+1}}|A_i|\right)+q^{d+1-k}\sum_{j=1}^{k-1}\left(\prod_{\substack{1\leq i\leq k+1\\i\neq j, j+1, j+2}}|A_i|\right)(|A_j||A_{j+2}|)^{\frac12}+\\
 &q^{d+1-k}\sum_{\substack{1\leq j_1<j_2\leq k\\j_1\leq j_2-2}}\left(\prod_{\substack{1\leq i\leq k+1\\i\neq j_1, j_1+1, j_2, j_2+1}}|A_i|\right)(|A_{j_1}||A_{j_1+1}||A_{j_2}||A_{j_2+1}|)^{1/2}+\\
 &\sum_{m=3}^k\sum_{\substack{\vec{v}\in\{0,1\}^k\\wt(\vec{v})=m}}q^{d+m-1-k}\left(\prod_{\substack{1\leq i\leq k+1\\i\neq j_1, j_1+1,j_m, j_m+1 }}|A_i|\right)|A_{j_1+1}|^{\frac12}|A_{j_m}|^{\frac12}|A_{j_1}|^{\frac12}|A_{j_m+1}|^{\frac12}.
  \end{split}
\end{equation}
If there is a constant $C(k)$ only dependent on $k$ such that $|A_i||A_{i+1}|\geq C(k)q^{d+k-1}$ (for instance, one can choose $C(k)=(10k!)^2$), then we have
$$
|R_0|>|R_1|+|R_2|+\cdots+|R_k|.
$$
So
$$
|\Pi_{\vec{\alpha}, G}(A_1\times A_2\times \cdots\times A_{k+1})| =R_0+R_1+\cdots+R_k\geq|R_0|-(|R_1|+|R_2|+\cdots+|R_k|)>0.
$$
Since $|\Pi_{\vec{\alpha}, G}(A_1\times A_2\times \cdots\times A_{k+1})|$ is a nonnegative integer by its definition, it follows that
$$
|\Pi_{\vec{\alpha}, G}(A_1\times A_2\times \cdots\times A_{k+1})|\geq1,
$$
i.e. $\vec{\alpha}\in\Pi_G(A_1\times A_2\times \cdots\times A_{k+1})$. This holds for every $\vec{\alpha}\in(\mathbb{F}_q^*)^k$. Thus $\Pi_G(A_1\times A_2\times \cdots\times A_{k+1})\supseteq(\mathbb{F}_q^*)^k$.

\bibliographystyle{abbrv}
\bibliography{dot_product_REF}
\end{document}